\documentclass[preprint,11pt]{elsarticle}
\usepackage{lineno,hyperref} 
\modulolinenumbers[5]

\usepackage{amssymb,amsmath}
\usepackage{algorithm}
\usepackage{algorithmic}
\usepackage{epsfig}
\usepackage{subfigure}
\usepackage{multirow}
\usepackage{float}
\usepackage{longtable}
\usepackage{booktabs}
\usepackage{color}
\usepackage{subfigure}
\usepackage{epstopdf}
\usepackage{ulem}
\usepackage{bm}
\usepackage{amsopn}
\usepackage{amssymb,amsmath,color,times}
\usepackage{lineno}
\usepackage{array}
\newenvironment{shrinkeq}[1]
{ \bgroup
  \addtolength\abovedisplayshortskip{#1}
  \addtolength\abovedisplayskip{#1}
  \addtolength\belowdisplayshortskip{#1}
  \addtolength\belowdisplayskip{#1}}
{\egroup\ignorespacesafterend}

\usepackage{amsfonts,epsfig}
\usepackage{geometry}
\usepackage{fancyhdr}                           
\numberwithin{equation}{section}       
\usepackage{enumerate}
\usepackage{mdframed}                          
\usepackage{color}

\usepackage{amsthm}                	
\theoremstyle{plain}	
\newtheorem{theorem}{Theorem}[section]
\newtheorem{lemma}{Lemma}[section]
\newtheorem{corollary}{Corollary}[section]
\theoremstyle{definition}
\newtheorem{example}{Example}[section]

\numberwithin{algorithm}{section}

\begin{document}
\begin{frontmatter}
\title{Randomized QLP algorithm and error analysis \tnoteref{mytitlenote}}

\author[mymainaddress]{Nianci\ Wu}
\cortext[mycorrespondingauthor]{Corresponding author}

\author[mymainaddress]{Hua\ Xiang\corref{mycorrespondingauthor}}
\ead{hxiang@whu.edu.cn}

\address[mymainaddress]{School of Mathematics and Statistics, Wuhan University, Wuhan 430072, PR China.}

\begin{abstract}
In this paper, we describe the randomized QLP (RQLP) algorithm  and its enhanced  version (ERQLP) for computing the low rank approximation to $A$ of size $m\times n$  efficiently such that $A\approx QLP$, where $L$  is the rank-$k$  lower-triangular matrix, $Q$ and $P$ are column orthogonal matrices. The theoretical cost of the implementation of  RQLP and ERQLP only needs $\mathcal{O}(mnk)$. Moreover, we derive the upper bounds of the expected approximation error $\mathbb{E}\left [  (\sigma_{j}(A) - \sigma_{j} (L))/ \sigma_{j}(A) \right] $ for $j=1,\cdots, k$, and prove  that the $L$-values of the proposed methods  can track the singular values of $A$ accurately. These claims are supported by extensive numerical experiments.
\end{abstract}


\end{frontmatter}

\section{Introduction}\label{sec1}\noindent
Given an $m\times n$ matrix $A$,  consider a randomized algorithm for computing rank-$k$ ($k\leq {\rm rank}(A)\leq n \leq m$) approximate decomposition efficiently
\begin{equation}
  A \approx Q L  P^T,
\end{equation}
where $Q$ and $P$ are column orthogonal matrices and $L$  is lower-triangular matrix. The diagonal elements of $L$ are called {\it $L$-values}.

One way for gaining SVD-type information is the column pivoted QR (CPQR) decomposition,  which  admits the form
$A\Pi = QR$,  where $\Pi $ is a  permutation matrix,  factor $R$  is a upper-triangular matrix and  the diagonal elements of $R$ are called the {\it $R$-values},  and $Q$  is orthogonal matrix.
The  partial CPQR can compute a rank-$k$ approximate column subspace of $A$ spanned by the leading $k$ columns in $A\Pi$.
To improve the low-rank approximation, specialized pivoting strategies will result in some {\it rank-revealing} QR factorizations, see \cite{Chan1987, Demmel2015,Golub2013, Gu1996} and the $R$-values can roughly approximate the singular values of $A$ \cite{Huckaby2002,Huckaby2005}.

As another  candidate of the SVD,  G. W. Stewart \cite{Stewart1999} proposed a so called  pivoted QLP decomposition. The algorithm consists of  two  CPQR  decompositions. To be precise,  the matrix $A\Pi_0$ is first factored as $A\Pi_0 = \hat Q R$, then the matrix $R^T\Pi_1$ is factored as $R^T\Pi_1  = \hat P L^T$, resulting in
\begin{equation}\label{pivotedqlp}
  A =  Q L P^T = \hat Q\Pi_1 L  \hat P^T\Pi_0^T
\end{equation}
 with $ Q = \hat Q\Pi_1$ and $ P = \Pi_0 \hat P$, where $Q$ and $P$ are orthogonal, $\Pi_0$  and $\Pi_1$  are permutation matrices,  and $L$ is lower triangular.
  It was shown that  the $L$-values can track the singular values of $A$ far better than the $R$-values, see  \cite{Huckaby2005}.  Similarly, the truncated pivoted QLP decomposition can also be used as a low rank approximation of $A$.



 We see that  the pivoted QLP decomposition applies twice pivoted QR decomposition, once to $A$ and once to $R^T$.  The Householder triangularization with column pivoting requires about $(2mn^2 - 2n^3/3)$  floats when  matrix $A$ is of full-rank. The additional work in the reduction to lower triangular form is $4n^3/3$ floats, see \cite{Golub2013}. Thus when $m=n$ the pivoted QLP decomposition requires twice as much work as the pivoted QR decomposition. Therefore, it is a motive of this paper to design a fast algorithm to reduce the computation cost of the pivoted QLP decomposition without losing too much accuracy.

Our paper is organized as follows.  In Section \ref{sec2}, we  propose the randomized QLP decomposition algorithm and its enhanced  version. Next,  the asymptotic convergence  rates of $L$-values are discussed in Section \ref{sec3}.
Then, we give the  block version of  randomized QLP decomposition algorithm in Section \ref{sec4}. Finally, we provide the numerical experiments to confirm the effectiveness of proposed algorithms in Section \ref{sec5} and some conclusion  remarks in last section.

Throughout the paper $\sigma_i(M)$ will denote the $i$-th singular value of $M$ in descending order, and $\sigma_{\min}(M)$ is the smallest singular value of $M$. $\| M \|_F$ denotes the Frobenius norm, $\| M \|_2 = \sigma_1(M)$  is the  spectral norm.

\section{Randomized QLP decomposition}\label{sec2}\noindent
Let $V$ be a given $m\times k $~$(k\leq n\leq m)$ matrix with independent columns, the  orthogonal projector of the {\it R($V$)} is defined by
$$P_V= V(V^TV)^{-1}V^T, $$
where {\it R($X$)} denotes the range of a matrix $X$.  Especially, when $V$ is orthonormal, $P_V = VV^T$.  It is easy to verify \cite{Gu2016} that
\begin{equation*}
  (A-P_VA)^T (A-P_VA) = A^T A - B^T B,
\end{equation*}
where $B=V^TA$. Based on the fact that for a matrix $M$, $\| M \|_F^2 = tr(M^T M)$,  where $tr(X)$ means the trace of a matrix $X$, thus
\begin{equation}\label{fixedrank}
  \| A - P_VA\|_F^2  = \|A \|_F^2  - \| B\|_F^2.
\end{equation}
Suppose the orthogonal basis vectors of $k$-dimensional dominant subspace of $A$ form the orthonormal matrix $V$, then we have $ A \approx P_VA$. It also can be written as
\begin{equation}\label{rankkapp2}
 \| A -  P_VA\|_{F} \approx \min\limits_{{\rm rank}(X)\leq k}\| A - X\|_F.
\end{equation}
There are similar results in the 2-norm.

Randomized range finder \cite{Halko2011} is usually a crucial step in the randomized matrix approximation framework, which can be used to approximate the matrix $V$ above. The basic idea is to use random sampling to identify the subspace capturing the dominant actions of a matrix.  To illustrate, we present a basic {\it randomized range finder} scheme as follows, see \cite{Gu2014, Halko2011, Mahoney2011, Martinsson2016,  Woolfe2008}.
\begin{enumerate}
\addtolength{\itemsep}{ -0.8 em} 
     \item  Draw a Gaussian random matrix $\Omega$ of size $n\times \ell$.
     \item  Compute a  sampling matrix $Y=A\Omega$ of size $m\times \ell$.
     \item  Orthonormalize the columns of $Y$ to form the $m\times \ell$ matrix $V$, e.g., using the QR, SVD, etc.
 \end{enumerate}
The number of columns $\ell$ is usually slightly larger than the target rank $k$ because we can obtain more accurate approximations of this form. We refer to this discrepancy $p =\ell-k$ as the over-sampling parameter. Usually,  $p=5$ or $p=10$ is often sufficient, see \cite{Halko2011}.
\subsection{The randomized QLP decomposition }\noindent
The randomized QLP (RQLP) decomposition can be split into two computational stages. The first is to construct a low-dimensional subspace and build an orthonormal matrix $V$ of size $m\times\ell$ that captures the action of $A$ via a randomized range finder. The second is to restrict the matrix to the subspace and then compute the QLP decomposition of the reduced matrix $B = V^T A$ , as described below.
\begin{enumerate}
\addtolength{\itemsep}{ -0.8 em} 
     \item Set $B = V^{T} A$ so that $P_VA = VB$.
     \item Compute the pivoted QLP decomposition  so that
      $B = Q_B L P_B^T$.
     \item Set $Q = V Q_B$, and $P = P_B$.
 \end{enumerate}
The pseudo-codes of rank-$k$ RQLP  algorithm is given in Algorithm \ref{RQLP}.
\begin{algorithm}[!h]
\caption{{\sc Randomized QLP decomposition Algorithm}}\label{RQLP}
{\bf Input: }{\it $A \in \mathbb{R}^{m\times n}$, a target rank $k\geq 2$, an oversampling parameter $p\geq 2$.}\\
{\bf Output: } $[Q,~L,~P] = {\sf RQLP}(A,k,p)$.\\
\begin{shrinkeq}{-4.5ex}
\begin{flalign*}
 &1.~\Omega = {\sf rand}(n, \ell).                    &&   \%  {\sf~Gaussian~random~matrix}~\Omega~{\sf~is~of~size}~n\times \ell  \qquad \\
 &2.~Y = A\Omega.                                         &&                                               \\
 &3.~[V, \sim] =  {\sf  qr}(Y).                           &&   \%  {\sf~Range~finder}                    \\
 &4.~B=V^T A.                                                &&                                               \\
 &5.~[Q_0, R_0, \Pi_0] = {\sf cpqr}(B).            &&   \%  {\sf~The~first~CPQR~decomposition}  \\
 &6.~[Q_1, L^T, \Pi_1] = {\sf cpqr}(R_0^T).    &&   \%  {\sf~The~second~CPQR~decomposition} \\
 &7.~Q= VQ_0\Pi_1,~~P = \Pi_0 Q_1.           &&
\end{flalign*}
\end{shrinkeq}
\vspace{2pt}
\end{algorithm}

Assume that the error matrix is $E$ in the randomized range finder, i.e., $E = A - VV^T A$. Then we can have
\begin{equation}\label{error1}
  A = QLP^T + E.
\end{equation}
The bounds on the probability of a large deviation given in  \cite[Theorem 10.5]{Halko2011} for the randomized range finder can directly apply to the output of the RQLP algorithm too.
\begin{corollary}
  Under the hypotheses of Algorithm {\rm\ref{RQLP}}, construct the orthonormal matrix $V$ according to the randomized range finder and perform the pivoted QLP decomposition to the reduced matrix $B=V^TA$ such that $B = Q_B L P_B^T$. If $p>2$, then the expected approximation error
 \begin{equation}\label{ERRF3}
   \mathbb{E}\left [ \| A-QLP^T\|_{F}\right]
   \leq \left (   1 + \frac{k}{p-1}\right )^{1/2} \left (    \sum\limits_{j=k+1}^{n} \sigma_{j}^2 \right )^{1/2},
 \end{equation}
 where  $\mathbb{E}$ denotes expectation, $Q = V Q_B$  and $P = P_B$.
\end{corollary}
When the error matrix $E$ is measured in the spectral norm, as opposed to the Frobenius norm, the randomized scheme is slightly further removed from optimality. For details, see \cite[Theorem 10.6]{Halko2011}.
\subsection{The enhanced RQLP decomposition }\noindent
In the case that  the given matrix $A$ is singular,  the  CPQR decomposition in the first step is reasonable, because it orders the initial singular values of $A$ and moves the zeros to the bottom of the matrix, i.e.,
\begin{equation*}
A\Pi = Q_0  \begin{bmatrix}R_{0}\\ 0 \end{bmatrix} ,
\end{equation*}
where $R_{0} $ is upper-triangular matrix.
Like the QR algorithm for computing the SVD of a real upper triangular matrix \cite{Chandrasekaran1995},  Huckaby and Chan \cite{Huckaby2002} showed  that if we perform the {\it QLP iteration}, the $L$-values produced by QLP iteration can track the singular values of $A$ far better than the ones produced by pivoted QLP decomposition.

That is to say,  we  continue to take the QR decomposition with {\it no pivoting} of the transpose of the $R$ factor produced by the last step, i.e.,
\begin{equation}
R_{i-1}^T =  Q_{i} R_{i}, \quad i =1,2,\cdots,
\end{equation}
the error bounds $(\sigma_j(R_{11}^{(i)})^{-1} - \sigma_{j}(A)^{-1})/$$\sigma_{j}(A)^{-1}$  will be improved by a quadratic factor in each step for $j=1,2,\cdots,k$, where the $k\times k$ upper triangular matrix $R_{11}^{(i)}$  is the first diagonal block of $R_{i}$, see \cite{Huckaby2002}.

Inspired by this, we use the unpivoted inner QR decomposition to  $R_{0}$ in Algorithm \ref{RQLP},  and  we can get the enhanced randomized QLP algorithm, which is denoted by {\sf ERQLP}.
For $d$ times  inner QR decomposition, the pseudo-codes of ERQLP algorithm is  given in Algorithm \ref{ERQLP}, where $d$ is an even number. A similar argument can also be used when $d$ is odd.
\begin{algorithm}[!h]
\caption{{\sc Enhanced randomized QLP decomposition Algorithm}}\label{ERQLP}
{\bf Input: }{\it $A \in \mathbb{R}^{m\times n}$, a target rank $k\geq 2$, an oversampling parameter $p\geq 2$ and an even number of inner QR iterations $d$.}\\
{\bf Output: } $[Q,~L,~P] = {\sf ERQLP}(A,k,p,d)$.\\
\begin{shrinkeq}{-4.5ex}
\begin{flalign*}
&1.~\Omega = {\sf rand}(n, \ell).                                        && \% ~  {\sf Gaussian~random~matrix}~\Omega~{\sf is~of~size}~n\times \ell\\
&2.~Y = A\Omega.                                                              &&\\
&3.~[V, \sim] =  {\sf  qr}(Y).                                               && \% ~  {\sf Range~finder}\\
&4.~B=V^T A.                                                                     && \\
&5.~[Q^{(0)}, R^{(0)}, \Pi] = {\sf cpqr}(B).                        && \%~   {\sf   The~first~CPQR~decomposition}\\
&6.~{\sf for} ~i=1,\cdots,d                                                 && \%~ d {\sf ~times~inner~QR~decomposition}\quad\qquad\\
&7.~\quad [Q^{(i)},~ R^{(i)}] = {\sf qr}([R^{(i-1)}]^T).      && \%~   {\sf Inner~QR~decomposition}\\
&8.~{\sf end}                                                                      && \\
\end{flalign*}
\end{shrinkeq}
\vspace{-1pt}
9.~$Q = V Q^{(0)} Q^{(2)}\cdots Q^{(d)}$,~$L = [R^{(d)}]^T$,
  ~$P = \Pi Q^{(1)}Q^{(3)}\cdots Q^{(d-1)}$.
\vspace{2pt}
\end{algorithm}

\subsection{Computational complexity}\noindent
In this subsection,  we analyze the theoretical cost of the implementation of  Algorithm \ref{RQLP}--\ref{ERQLP}, and we compare it to those of  the truncated pivoted QLP decomposition.

Let $C_{{\sf mm}}$,  $C_{{\sf  qr}}$ and $C_{{\sf cpqr}}$ denote the scaling constants for the cost of executing the matrix-matrix multiplication, the full QR decomposition and the CPQR decomposition, respectively. More specifically, we assume that
\begin{enumerate}
\addtolength{\itemsep}{ -0.8 em} 
     \item multiplying two matrices of size $m\times n$ and $n\times r$ costs $C_{{\sf mm}}mnr$.
     \item performing a QR decomposition with no pivoting of a matrix of size $m\times n$, with $m\geq n$, costs $C_{{\sf  qr}}mn^2$.
     \item performing a CPQR decomposition of a matrix of size $m\times n$, with $m\geq n$, costs $C_{{\sf cpqr}}mn^2$.
 \end{enumerate}
The execution time for Algorithm \ref{RQLP} is easily seen to be
\begin{equation*}
  T_{{\sf RQLP}}  \sim 2 C_{{\sf mm}} mn\ell  + (C_{{\sf mm}}  +  C_{{\sf  qr}})m\ell^2 +  C_{{\sf  cpqr}} n\ell^2  +  C_{{\sf  cpqr}} \ell^3.
\end{equation*}
Similarly, the execution time for Algorithm \ref{ERQLP} is
\begin{equation*}
  T_{{\sf ERQLP}} \sim 2 C_{{\sf mm}} mn\ell  + (C_{{\sf mm}}  +  C_{{\sf  qr}})m\ell^2 +  C_{{\sf  cpqr}} n\ell^2 + d ( C_{{\sf mm}} + C_{{\sf  qr}})  \ell^3.
\end{equation*}
To compute the rank-$\ell$  pivoted QLP decomposition, we require  to apply  partial CPQR decomposition on the $m$-by-$n$ matrix $A$, i.e., $ A \Pi(:, 1:\ell) =  Q R $,  and compute the full CPQR decomposition of the $n$-by-$\ell$ matrix $\begin{bmatrix} R_{11} & R_{12}  \end{bmatrix}^T$. This gives  a total execution time of
\begin{equation*}
T_{{\sf QLP}} \sim   C_{{\sf  cpqr}} mn\ell  + C_{{\sf  cpqr}} n\ell^2.
\end{equation*}
Compared with  $T_{{\sf QLP}}$ by omitting the lower order terms,  we see  that  $T_{{\sf RQLP}}$ and $T_{{\sf ERQLP}}$   have the same order of magnitude in cost flops. But the truncated QLP is very time-consuming to permute the data required in CPQR for the $m$-by-$n$ matrix $A$, while the most flops for RQLP and ERQLP  are spent on the matrix-matrix multiplications which is the so-called nice BLAS-3 operations, and the original large-scale matrix $A$ is visited only twice.


\section{Error analysis}\label{sec3}\noindent
In this section, we assess the ability of the approximate QLP decomposition, i.e., RQLP and ERQLP, to capture the singular values. We start with the following lemmas.
\begin{lemma}\label{Thm31}{\rm \cite[Theorem 3.1]{Gu2014}}
Let $A$ be an $m$-by-$n$ matrix and $V$ be a matrix with orthonormal columns. Then $\sigma_j(A) \geq \sigma_j(V^T A)$ for $1\leq j \leq \min(m,n)$.
 \end{lemma}

\begin{lemma}\label{Thm56}{\rm \cite[Theorem 5.6]{Gu2014}}
Under the hypotheses of Algorithm {\rm\ref{RQLP}}, construct the column orthogonal matrix $V$ and the reduced matrix $B = V^T A$ according to the randomized range finder, and let $B_k$ be the rank-$k$ truncated SVD of $B$. Then for any $p \geq 2$,
\begin{equation*}\label{ERRF3}
   \mathbb{E}\left [ \sigma_{j} (VB_k)\right]\geq \frac{\sigma_{j}(A)}{  \sqrt{1 + \mathcal{C}^2 \tau_j^2} },
 \end{equation*}
 where $\mathcal{C} = 4e\sqrt{\ell} (\sqrt{n-\ell + p } + \sqrt{\ell} + 7)$, $\tau_j = \sigma_{k +1}(A)/\sigma_{j}(A)$, $j=1,\cdots,k$.
\end{lemma}
For all $ 1 \leq j \leq k$, Lemma \ref{Thm56} implies that
\begin{equation}\label{ERRF3}
   \mathbb{E}\left [ \frac{\sigma_{j}(A) - \sigma_{j} (VB_k)}{ \sigma_{j}(A)}\right] \leq    1 - \frac{1}{ \sqrt{1 + \mathcal{C}^2 \tau_j^2} }.
 \end{equation}
Then we have the following theorem to  illustrate that the L-values of RQLP can track the singular values of $A$ precisely.
\begin{theorem}\label{Thm3.1}
 Under the hypotheses of Algorithm {\rm\ref{RQLP}}, compute the pivoted QLP decomposition so that $B = Q_B L P_B^T$ and partition the lower triangular matrix $L$ into diagonal blocks $L_{11}$ and $L_{22}$ and off-diagonal block $L_{21}$, where $L_{11}$ is of size $k$-by-$k$. Assume that the bounds
\begin{equation}\label{bound1}
\gamma  \geq \sigma_k(B)/\sqrt{k(n-k+1)},\quad \| L_{22} \|_2\leq \sigma_{k+1}(B)\sqrt{(k+1)(n-k)}
\end{equation}
 hold, where $\gamma = \sigma_{\min}(L_{11})$  and $\rho=\| L_{22} \|_2 /\gamma < 1$. Then for $ 1 \leq j \leq k$,
 \begin{equation}\label{Thm1}
\mathbb{E}\left [ \frac{\sigma_{j}(A) - \sigma_{j} (L_{11})}{ \sigma_{j}(A)}\right] \leq
1 - \frac{1}{ \sqrt{1 + \mathcal{C}^2 \tau_j^2} }  + \mathcal{O}\left ( \frac{  \| L_{21}\|_2^2  }{  (1 - \rho^2) \gamma^2} \right ),
\end{equation}
where $\tau_j = \sigma_{k +1}(A)/\sigma_{j}(A)$.
\end{theorem}

\begin{proof}
The rank revealing QR decomposition provides bounds on the singular values of $B$ in terms of the norms of the blocks. By  Theorem 3.1 of Reference \cite{Mathias1993}, when the bounds \eqref{bound1} hold, we can have
\begin{equation*}
\frac{\sigma_j(L_{11})}{ \sigma_{j}(B) } \geq \left [   1 - \frac{  \| L_{21}\|_2^2  }{  (1 - \rho^2) \gamma^2  }  \right ]^{1/2}
\end{equation*}
for $ 1 \leq j \leq k$, which is equivalent to
\begin{equation*}
\frac{\sigma_j(L_{11})}{ \sigma_{j}(B) } \geq 1 -   \mathcal{O}\left ( \frac{  \| L_{21}\|_2^2  }{  (1 - \rho^2) \gamma^2} \right).
\end{equation*}
The above relation  can be further rewritten as
\begin{equation}\label{sigamBL2}
\sigma_{j}(B) -  \sigma_j(L_{11})  \leq \sigma_{j}(B)  \mathcal{O}\left ( \frac{  \| L_{21}\|_2^2  }{  (1 - \rho^2) \gamma^2} \right ).
\end{equation}
Let $V$  be the column orthogonal matrix constructed by Algorithm {\rm\ref{RQLP}}. Using Lemma \ref{Thm31}, we have
\begin{equation*}
\sigma_{j}(B) = \sigma_j(V^T A) \leq \sigma_j(A).
\end{equation*}
From this,  the formula \eqref{sigamBL2}  is reduced to
\begin{equation}\label{sigamBL3}
\frac{\sigma_{j}(B) -  \sigma_j(L_{11})}{\sigma_{j}(A)}  \leq \mathcal{O}\left ( \frac{  \| L_{21}\|_2^2  }{  (1 - \rho^2) \gamma^2} \right ).
\end{equation}
On the other hand,
 \begin{equation*}
   \begin{aligned}
   \mathbb{E}\left [   \sigma_j(A) - \sigma_j (L_{11}) \right]
   &  =  \mathbb{E}\left [ \sigma_j(A)  - \sigma_j(VB_k) \right] +  \mathbb{E}\left [ \sigma_j(VB_k)  -  \sigma_j (L_{11}) \right]\\
   &  =  \mathbb{E}\left [ \sigma_j(A)  - \sigma_j(VB_k) \right] +  \mathbb{E}\left [ \sigma_j(B)  -  \sigma_j (L_{11}) \right].
   \end{aligned}
\end{equation*}
Thus,
\begin{equation*}
\mathbb{E}\left [ \frac{\sigma_{j}(A) - \sigma_{j} (VB_k)}{ \sigma_{j}(A)}\right] =
\mathbb{E}\left [ \frac{\sigma_{j}(A) - \sigma_{j} (VB_k)}{ \sigma_{j}(A)}\right] +
\mathbb{E}\left [ \frac{\sigma_{j}(B) - \sigma_{j}(L_{11})}{\sigma_{j}(A)}\right].
\end{equation*}
Combing \eqref{ERRF3} and \eqref{sigamBL3}, we arrive at \eqref{Thm1}.
\end{proof}

Next, we can  obtain the error upper bounds for the $L$-values of ERQLP  similarly.
\begin{lemma} {\rm \cite[Theorem 4.1]{Huckaby2002}}
Let $B$ be an $\ell $-by-$n$ matrix, $n\geq \ell \geq k$ and the singular values of $B$ satisfy $\sigma_{k}(B)>\sigma_{k+1}(B)$. The upper triangular matrix $R^{(0)}$ is the $R$-factor in the first  pivoted QR decomposition of $B$, i.e., $B\Pi = Q^{(0)} R^{(0)}$  and then apply $i$ times inner unpivoted QR decomposition to $[R^{(i-1)}]^T$, i.e., $[R^{(i-1)}]^T = Q^{(i)}R^{(i)}$ ($i=1,2,\cdots$).  Partition  the upper triangular matrix $R^{(i)}$ into diagonal block
$R_{11}^{(i)}$ and $R_{22}^{(i)}$ and off-diagonal block $R_{12}^{(i)}$, where $R_{11}^{(i)}$ is of size $k$-by-$k$. Assume the bounds
  \begin{equation}\label{LRbound}
  \| R_{22}^{(0)}\|_2 \leq \sqrt{(k+1)(n-k)} \sigma_{k+1}(B),~\gamma^{(0)}\geq \frac{\sigma_k}{\sqrt{k(n-k+1)}},~\rho^{(i)}  <1
\end{equation}
hold, where $\gamma^{(i)} = \sigma_{\min}(R_{11}^{(i)})$ and $\rho^{(i)} = \| R_{22}^{(i)}\|_2/\sigma_{\min}(R_{11}^{(i)})$, then for $j=1,\cdots, k$,
\begin{equation}\label{Itebound0}
  \frac{\sigma_j(R_{11}^{(i)})^{-1} - \sigma_{j}(B)^{-1}}{\sigma_{j}(B)^{-1}} \leq \left(\frac{\sigma_{k+1}(B)}{\sigma_{k}(B)}\right)^{2i}
  \mathcal{O}\left ( \frac{n^{(4i+1)/2}  \| R_{22}^{(0)}\|_2}{ [1-(\rho^{(i)})^2] (\gamma^{(i)})^2} \right ).
\end{equation}
\end{lemma}
From the interlacing property of singular values, $\sigma_{j}(B)  \geq \sigma_j(R_{11}^{(i)})$. The formula \eqref{Itebound0} implies that
\begin{equation}\label{Itebound1}
  \frac{\sigma_{j}(B)  - \sigma_j(R_{11}^{(i)}) }{\sigma_{j}(B) } \leq \left(\frac{\sigma_{k+1}(B)}{\sigma_{k}(B)}\right)^{2i}
  \mathcal{O}\left ( \frac{n^{(4i+1)/2}  \| R_{22}^{(0)}\|_2}{ [1-(\rho^{(i)})^2] (\gamma^{(i)})^2} \right ).
\end{equation}
It means that if the relative error bound  $ [\sigma_{j}(B)  - \sigma_j(R_{11}^{(i)}) ]/\sigma_{j}(B)  $  is improved with a quadratic factor  by one inner QR decomposition.

\begin{theorem}\label{Thm2}
 Under the hypotheses of Algorithm {\rm\ref{ERQLP}}, $B = Q_B L P_B^T$, where
  $Q_B = Q^{(0)} Q^{(2)}\cdots Q^{(d)}$, $L = [R^{(d)}]^T$,  $P_B =  \Pi Q^{(1)}Q^{(3)}\cdots Q^{(d-1)}$ , and
    the upper triangular matrix $R^{(0)}$ and $R^{(i)}$ are the {\it R}-factor in the first  pivoted QR decomposition and the $i$-th inner QR decomposition, respectively, i.e.,
    \begin{equation*}
      B \Pi = Q^{(0)} R^{(0)},\quad [R^{(i-1)}]^T = Q^{(i)} R^{(i)},
    \end{equation*}
where $i=1,\cdots,d$. Let $L$  and $R^{(d)}$ be  partitioned as
  \begin{equation*}\label{partitial1}
  L       =  \begin{bmatrix}  L_{11}      &    0         \\ L_{12}       & L_{22}        \end{bmatrix},\quad
  R^{(d)} =  \begin{bmatrix}  R_{11}^{(d)}& R_{12}^{(d)} \\      0       & R_{22}^{(d)}  \end{bmatrix}.
\end{equation*}
Assume that the bounds in  \eqref{LRbound} hold. Then for $ 1 \leq j \leq k$,
 \begin{equation}\label{Itebound2}
\mathbb{E}\left [ \frac{\sigma_{j}(A) - \sigma_{j} (L_{11})}{ \sigma_{j}(A)}\right] \leq
1 - \frac{1}{\sqrt{1 + \mathcal{C}^2\tau_j^2}}  + \left (  \frac{\sigma_{k+1}(B)}{\sigma_{k}(B)}                                                                                                  \right)^{2d}
                                               \mathcal{O}\left (  \frac{ n^{(4d+1)/2} \| R_{12}^{(0)}\|_2^2}{\left  [1 - (  \rho^{(d)}  )^2 \right ] (\gamma^{(d)})^2} \right ).
\end{equation}
\end{theorem}
\begin{proof}
  The proof is similar to those given in Theorem \ref{Thm3.1}. Based on \eqref{ERRF3} and \eqref{Itebound1}, we can get \eqref{Itebound2} directly.
\end{proof}
From Theorem \ref{Thm2},  one can observe that for $j=1,2,\cdots,k$,  the larger the iteration number $d$  in Algorithm \ref{ERQLP} is,  the smaller  the expected relative error bound  $ [\sigma_{j}(A)  - \sigma_j(L_{11}) ]/\sigma_{j}(A)$ is.  The  $L$-values from ERQLP will be the better approximations to  the singular values of $A$   than those of RQLP algorithm. But at the same time,  the number of operations at each iteration will increase. A lot of numerical tests show that $d=2$ or $4$ is an appropriate choice.

\section{The block version of RQLP}\label{sec4}\noindent
It is straightforward to convert the RQLP algorithm to a block scheme, denoted by the block RQLP (BRQLP) algorithm.  Suppose the target rank $\ell$ and the block size $b$ satisfy $\ell=hb$ for some integer $h$.  Partition the Gaussian random matrix  $\Omega$ into slices $\{ \Omega_j \}_{j=1}^h$, each of size is $n$-by-$b$,  so that  $\Omega = [ \Omega_1,~\Omega_2,\cdots,~\Omega_h]$,  $Y = A\Omega = [ Y_1,~Y_2,\cdots,~Y_h]$, where $Y_j = A \Omega_j$ for $j=1,\cdots,h$. Analogously,  partition the orthonormal matrix $V$ and the reduced matrix $B$ in groups of $b$ columns and $b$ rows, respectively,
\begin{equation*}
  V = [ V_1,~V_2,\cdots,~V_h],\quad   B =
  \begin{bmatrix}
 B_1\\\vdots \\ B_h
  \end{bmatrix}.
\end{equation*}
Inspired by the block version of the randomized range finder in \cite{Martinsson2016}, we first initiate the algorithm by setting $A^{(0)} = A$ and construct the matrices $\{ V_j \}_{j=1}^h$ and $\{ B_j \}_{j=1}^h$ one at a time.

We further  compute the unpivoted QR decomposition of $Y_j  $  because the R-factor is never used, i.e.,
\begin{equation}
  [V_j,\sim] = {\sf qr}(Y_j).
\end{equation}
Just like the classical (non-block) Gram-Schmidt orthogonalization procedure \cite{Demmel1997}, the round-off errors will cause loss of orthonormality among the columns in $[ V_1,~V_2,~\cdots,~V_j]$ . The remedy is to adopt the reorthogonalized strategy.
\begin{equation}
   [V_j,\sim] = {\sf qr}(V_j   -   \Sigma_{i=1}^{j-1} V_i V_i^T V_j).
\end{equation}
Next, calculate the  block reduced matrix $B_j$  and  $A^{(j)}$ by
\begin{equation}\label{BVA}
   B_j = V_j^TA^{(j-1)}, \quad A^{(j)} = A^{(j-1)} - V_j B_j .
\end{equation}
Then apply the standard pivoted QLP decomposition to $B_j$ such that
\begin{equation}
  B_j   = Q_B^{(j)}  L_j [P_B^{(j)} ]^T ,
\end{equation}
 and update $Q_j = V_j Q_B^{(j)}$, ~$P_j = P_B^{(j)}$.   To be precise,  the pseudo-codes of BRQLP algorithm is shown in Algorithm \ref{BRQLP}
 \begin{algorithm}[!h]
\caption{{\sc Block Randomized QLP decomposition Algorithm}}\label{BRQLP}
{\bf Input: }{\it $A \in \mathbb{R}^{m\times n}$, a target rank $k\geq 2$, an oversampling parameter $p\geq 2$, a block size $b$.}\\
{\bf Output: } $[Q,~L,~P] = {\sf BRQLP}(A,k,p,b)$.\\
\begin{shrinkeq}{-4.5ex}
\begin{flalign*}
&1.~\Omega = {\sf rand}(n, \ell).                                                                  &&\%~{\sf Gaussian~random~matrix}~ \Omega {\sf~is~of~size}~n\times \ell\qquad\\
&2.~{\sf for}~j=1, \cdots, h                                                                          && \\		
&3.~\quad  \Omega_j = \Omega(:,~(j-1)b + 1 : jb),~ Y_j = A\Omega_j.       && \\
&4.~\quad  [V_j,\sim] = {\sf qr}(Y_j).		                                                      && \%~{\sf Range~finder}\\
&5.~\quad  [V_j,\sim] = {\sf qr}(V_j-\Sigma_{i=1}^{j-1} V_i V_i^T V_j).	  && \%~{\sf Reorthogonalized~process}\\            	
&6.~\quad  B_j=V^T_j A^{(j-1)}.		                                                              && \\
&7.~\quad  A^{(j)} = A^{(j-1)} - V_j B_j.		                                              && \\
&8.~\quad  [Q_B^{(j)},~L_j,~P_B^{(j)}] = {\sf qlp}(B_j).                            && \%~{\sf Standard~QLP~decomposition} \\          	 	
&9.~\quad  Q_j = V_j Q_B^{(j)},~P_j = P_B^{(j)}.		                                  &&\\
&10.~{\sf end}    \\
\end{flalign*}
\end{shrinkeq}
\vspace{2pt}
11.~$Q = [ Q_1,~Q_2,\cdots,~Q_h]$, $L = {\sf blkdiag}(L_1,~L_2,\cdots,~L_h)$, $P = [ P_1,~P_2,\cdots,~P_h]$.
\vspace{2pt}
\end{algorithm}

From \eqref{BVA} we can have  $B_j = V_j^TA$. It follows that the error resulting from Algorithm \ref{BRQLP} is
\begin{equation}
\begin{aligned}
    A - QLP^T &  =  A - [ Q_1,\cdots,~Q_h] \begin{bmatrix}L_1 &&\\ & \ddots&\\&&L_s\end{bmatrix} [ P_1,~P_2,\cdots,~P_h]^T \\
                       & =  A - [ V_1,\cdots,~V_h]  \begin{bmatrix}B_1 \\  \vdots \\B_h\end{bmatrix} = A - VB = A - VV^TA.\\
\end{aligned}
\end{equation}
As shown in \cite{Martinsson2016},  for a fixed Gaussian matrix $\Omega$, the projectors from the block and unblock version of the randomized range finder are identical. Thus, the average-case analysis of Algorithm \ref{RQLP} can be similarly applied to  Algorithm \ref{BRQLP}.

For the block  randomized QLP algorithm, we assume that it stops after $s$ steps. Then, the runtime of BRQLP decomposition is
 \begin{equation*}
 \begin{aligned}
  T_{ \sf BRQLP}  &\sim  \sum_{j=1}^h  \left[  3C_{\sf  mm} mnb + 2C_{\sf qr} mb^2  + 2(j-1)C_{\sf  mm} mb^2 +  C_{\sf  cpqr}  nb^2  + C_{\sf  mm} mb^2 \right]\\
                              & \sim  3C_{\sf  mm} mn \ell  +   \frac{2}{h} C_{\sf qr} m\ell^2   +  C_{\sf  mm} m\ell ^2   + \frac{1}{h}  C_{\sf  cpqr}  n\ell^2  + \frac{1}{h}  C_{\sf  mm}  m\ell^2.
 \end{aligned}
\end{equation*}
Comparing with RQLP, we see that the BRQLP involves one additional term of $C_{\sf  mm} mn \ell$, but needs less time to execute full QR decomposition.
\section{Numerical experiments}\label{sec5}\noindent
In this section,  we give several examples to illustrate that the randomized algorithms are as accurate as the classical methods. All experiments are carried out on a Founder desktop PC with Intel(R) Core(TM) i5-7500 CPU 3.40 GHz by MATLAB R2016(a) with a machine precision of $10^{-16}$.
\begin{example}
 Specifically,  like Tropp in \cite{Tropp2017}, two  synthetic matrices are given by its SVD forms:
  \begin{equation*}
        A = U^{(A)}  \Sigma^{(A)}  (V^{(A)})^{T}\in \mathbb{R}^{n\times n},
\end{equation*}
where $U^{(A)}\in \mathbb{R}^{n\times n}$ and $V^{(A)}\in \mathbb{R}^{n\times n}$ are random orthogonal matrices, the $\Sigma^{(A)}$ is diagonal with entries given by the following rules.
\begin{enumerate}[\sffamily I).]
\addtolength{\itemsep}{ -0.5 em} 
     \item Polynomially decaying spectrum~({\sf pds}):
       $\Sigma^{(A)} = {\rm diag}( 1,\cdots,1,2^{-s},3^{-s},\cdots,(n-t+1)^{-s})$,
     \item  Exponatially decaying spectrum~({\sf eds}):
       $ \Sigma^{(A)} = {\rm diag}(1,\cdots,1,2^{-s},2^{-2s},\cdots,2^{-(n-t)s})$,
   \end{enumerate}
 where the nonnegative constants $t$  and $s$ control the rank of the significant part of the matrix and  the rate of decay, respectively.
 \end{example}

\begin{example}
 The remaining  testing problems are from P. C. Hansen's  Regularization Tools (version 4.1) \cite{Hansen2008}. The ill-conditioned matrix $A$ are generated by the discretization of the Fredholm integral equation with  the first kind square integrable kernel
 \begin{equation*}
 \int_a^b \mathrm{K}(y,z)~ \textit{f}(z)~ dz = \textit{g}(y),  \quad c \leq y \leq d.
 \end{equation*}
When the Galerkin discretization method is used, we choose the examples {\sf heat} and {\sf  phillips}.
 \end{example}
In the following test,  we compare the proposed algorithms against several existing algorithms in terms of execution time and accuracy.  For more details,  we fix the matrix size $m = n$ and set the target rank $k=120$. The over-sampling parameter is $p=5$.  The approximate errors for the  singular values of $A$ are defined by
\begin{equation*}
 {\rm QLP/RQLP/ ERQLP:}~  err = \max\left\{   | \sigma_{j}(A) -  \sigma_{j}(L)  | :  j=1,\cdots,k    \right\}.
\end{equation*}

In Figure \ref{fig1} -- Figure\ref{fig2}, we display the  ability of $R$-values computed by CPQR and $L$-values computed by  QLP, RQLP and ERQLP, to capture the singular values of the input matrix as described in Section \ref{sec3}. It  shows  that the $L$-values can track the singular values of $A$ far better than the $R$-values.

 Table \ref{tab11} -- Table \ref{tab14} show the measured total CPU times $T$ (in seconds) and approximate error $err$, which lead us to make several observations:
 \begin{enumerate}[\sffamily (1).]
\addtolength{\itemsep}{ -0.5 em} 
     \item Comparing the accuracy of   QLP with its randomized variants, QLP and  RQLP have almost the same approximation error.  However, when two or four times inner QR iterations are taken, ERQLP performs a higher approximation accuracy than  RQLP and QLP.
     \item  Comparing the speed of   QLP with its randomized variants, randomized QLP algorithms  (RQLP and ERQLP) are decisively faster than QLP in all cases.  In more detail, we see that RQLP and ERQLP have the similar speed, with RQLP being slightly faster.
   \end{enumerate}

\begin{table}[!ht]
	\caption{Computational the CPU time and  error of different factorizations for the example {\sf pds}, where $t=30$ and $s=2$.}
	\centering
	\begin{tabular}{c ccc ccc ccc cc }
	\hline
    \multirow{2}{1cm}{{\rm n}} & \multicolumn{2}{l }{{\rm QLP}} &  &\multicolumn{2}{l }{{\rm RQLP}}& &\multicolumn{2}{l }{ERQLP~($d=2$)} & & \multicolumn{2}{l }{ERQLP~($d=4$)}\\
    \cline{2-3} \cline{5-6}\cline{8-9}\cline{11-12}
                  &  $T(s)$    & $err$     &  &  $T(s)$    & $err$   & &  $T(s)$    & $err$  & &  $T(s)$    & $err$                                  \\
     \hline
       2000    &  2.757  & 9.55e-02  &  &0.042         & 9.32e-02  & &0.078    & 3.58e-02   & & 0.082         & 2.50e-02        \\
       4000    &24.414  & 5.34e-02  &  &0.145         & 5.02e-02  & &0.282    & 5.20e-02   & & 0.267         & 2.97e-02      \\
       6000    &74.175  & 6.36e-02  &  &0.312         & 6.20e-02  & &0.551    & 2.80e-02   & & 0.554         & 2.09e-02       \\
     \hline
    \end{tabular}
\label{tab11}
\end{table}

\begin{table}[!ht]
	\caption{Computational the CPU time and  error of different factorizations for the example {\sf eds}, where $t=30$ and $s=1/20$.}
	\centering
	\begin{tabular}{c ccc ccc ccc cc }
	\hline
    \multirow{2}{1cm}{{\rm n}} & \multicolumn{2}{l }{{\rm QLP}} &  &\multicolumn{2}{l }{{\rm RQLP}}& &\multicolumn{2}{l }{ERQLP~($d=2$)} & & \multicolumn{2}{l }{ERQLP~($d=4$)}\\
    \cline{2-3} \cline{5-6}\cline{8-9}\cline{11-12}
                                                   &  $T(s)$    & $err$                            &  &  $T(s)$    & $err$                             & &  $T(s)$    & $err$                                 & &  $T(s)$    & $err$                                  \\
     \hline
       2000    & 2.741   & 1.65e-01  &  &0.041         & 1.68e-01  & &0.079     & 1.22e-01   & & 0.083     & 1.07e-02       \\
       4000    &11.901  & 1.76e-02  &  &0.145         & 1.75e-01  & &0.274     & 1.45e-01   & & 0.255      & 9.46e-02      \\
       6000    &75.822  & 1.69e-01  &  &0.277         & 1.65e-01  & &0.544     & 1.09e-01   & & 0.527      & 7.95e-02      \\
     \hline
    \end{tabular}
\label{tab12}
\end{table}

\begin{table}[!ht]
	\caption{Computational the CPU time and  error of different factorizations for the example {\sf heat}.}
	\centering
	\begin{tabular}{c ccc ccc ccc cc }
	\hline
    \multirow{2}{1cm}{{\rm n}} & \multicolumn{2}{l }{{\rm QLP}} &  &\multicolumn{2}{l }{{\rm RQLP}}& &\multicolumn{2}{l }{ERQLP~($d=2$)} & & \multicolumn{2}{l }{ERQLP~($d=4$)}\\
    \cline{2-3} \cline{5-6}\cline{8-9}\cline{11-12}
                                                   &  $T(s)$    & $err$                            &  &  $T(s)$    & $err$                             & &  $T(s)$    & $err$                                 & &  $T(s)$    & $err$                                  \\
     \hline
       2000    &4.039    & 8.62e-02  &  &0.086         & 8.62e-02  & &0.136          & 2.16e-02   & & 0.137         & 7.96e-03  \\
       4000    &27.572  & 8.62e-02  &  &0.235         & 8.62e-02  & &0.411          & 2.16e-02   & & 0.427         & 7.96e-03     \\
       6000    &95.160  & 8.62e-02  &  &0.567         & 8.62e-02  & &0.965          & 2.16e-02   & & 0.958         & 7.96e-03    \\
     \hline
    \end{tabular}
\label{tab13}
\end{table}

\begin{table}[!ht]
	\caption{Computational the CPU time and  error of different factorizations for the example {\sf phillips}.}
	\centering
	\begin{tabular}{c ccc ccc ccc cc }
	\hline
    \multirow{2}{1cm}{{\rm n}} & \multicolumn{2}{l }{{\rm QLP}} &  &\multicolumn{2}{l }{{\rm RQLP}}& &\multicolumn{2}{l }{ERQLP~($d=2$)} & & \multicolumn{2}{l }{ERQLP~($d=4$)}\\
    \cline{2-3} \cline{5-6}\cline{8-9}\cline{11-12}
                                                   &  $T(s)$    & $err$                            &  &  $T(s)$    & $err$                             & &  $T(s)$    & $err$                                 & &  $T(s)$    & $err$                                  \\
     \hline
       2000    &  3.194  & 7.12e-01  &  &0.086         & 7.10e-01  & &0.157          & 3.88e-01   & & 0.156         & 2.62e-01 \\
       4000    &25.026  & 7.12e-01  &  &0.264         & 7.06e-01  & &0.436          & 3.86e-01   & & 0.463         & 2.72e-01 \\
       6000    &91.232  & 7.12e-01  &  &0.563         & 7.08e-01  & &0.944          & 4.15e-01   & & 0.974         & 2.26e-01  \\
     \hline
    \end{tabular}
\label{tab14}
\end{table}

\begin{figure}[!h]
	\centering
	\subfigure{
		\includegraphics[height=5.2cm]{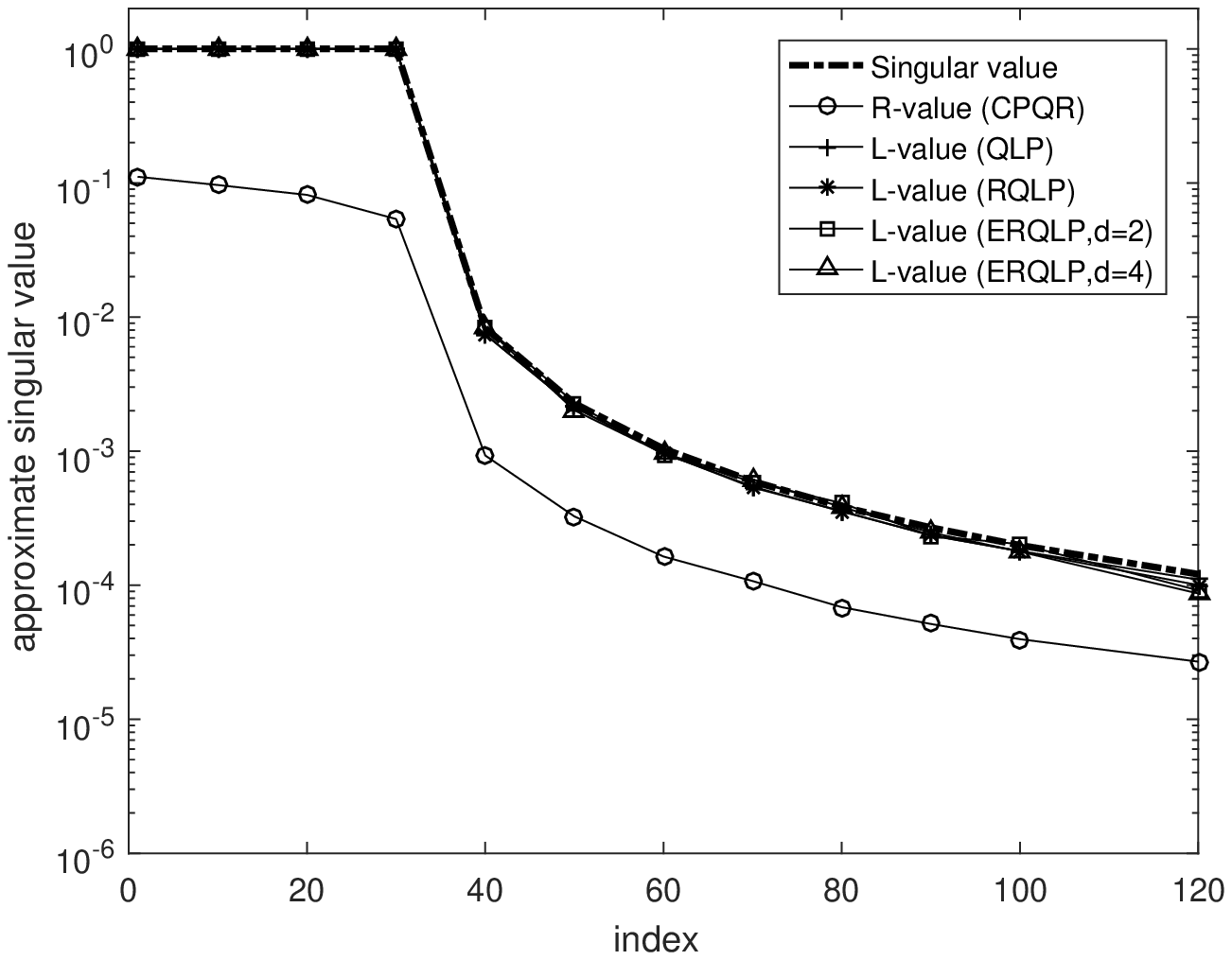}}
	\subfigure{
		\includegraphics[height=5.2cm]{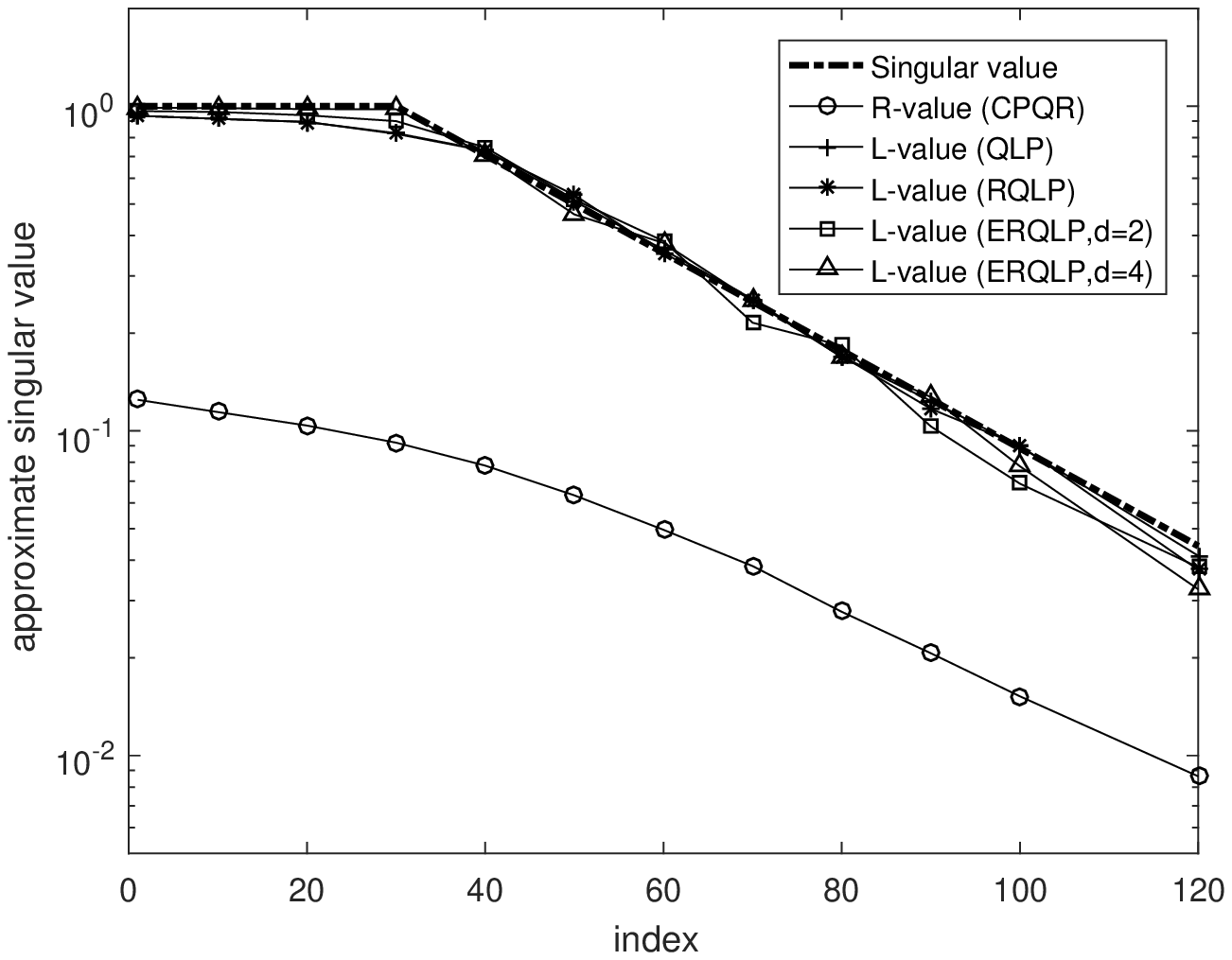}}
	\caption{The approximations to the singualr values of $A$ for the examples {\sf pds} (left) and {\sf eds} (right). Rank control parameter $t$ = 30. Rate of decay control parameters are set by $s$ = 2 and  $s$ = 1/20 for {\sf pds} and {\sf eds}, respectively. The matrix size is $4000 \times 4000$. }
	\label{fig1} 
\end{figure}

\begin{figure}[!h]
	\centering
	\subfigure{
		\includegraphics[height=5.2cm]{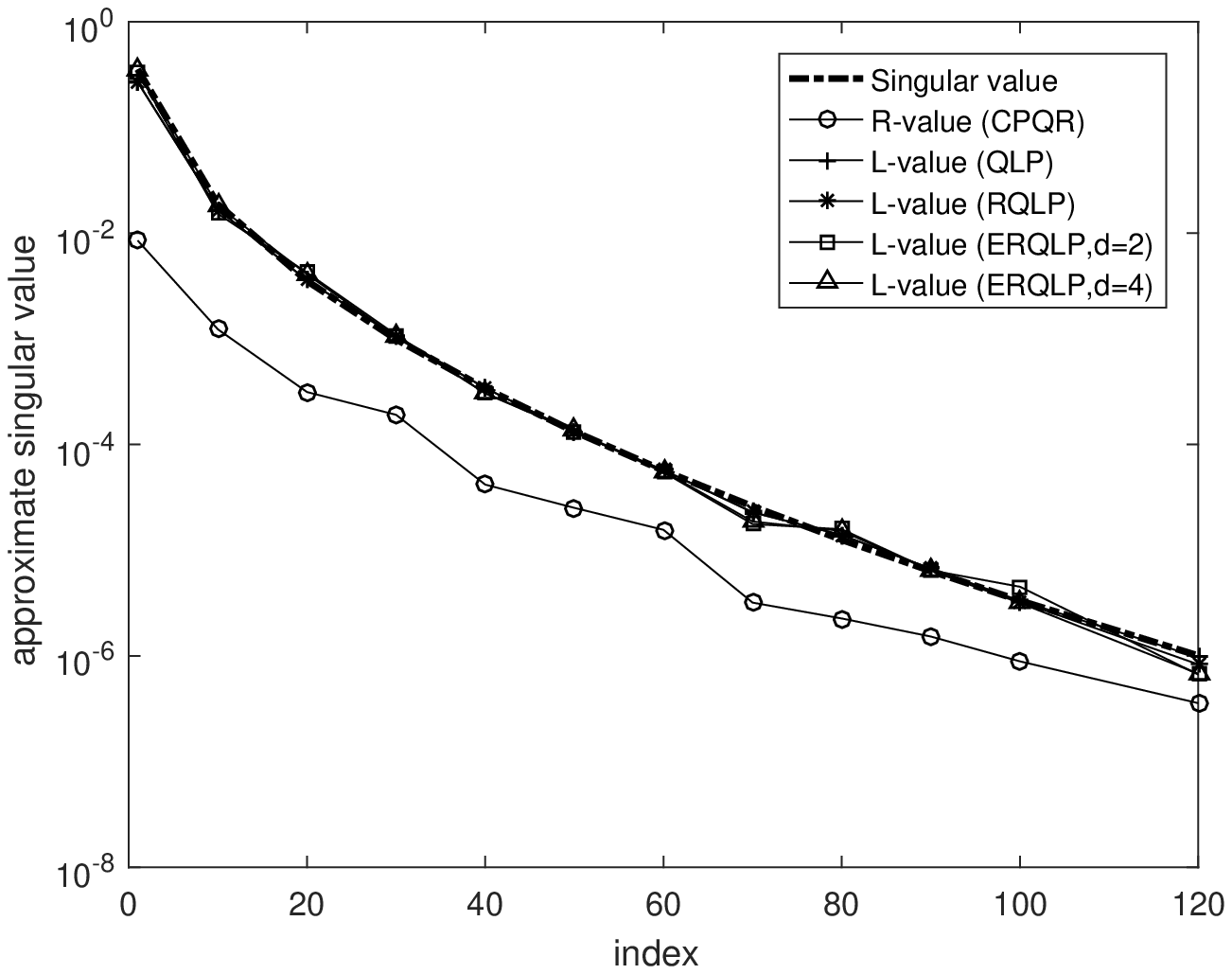}}
	\subfigure{
		\includegraphics[height=5.2cm]{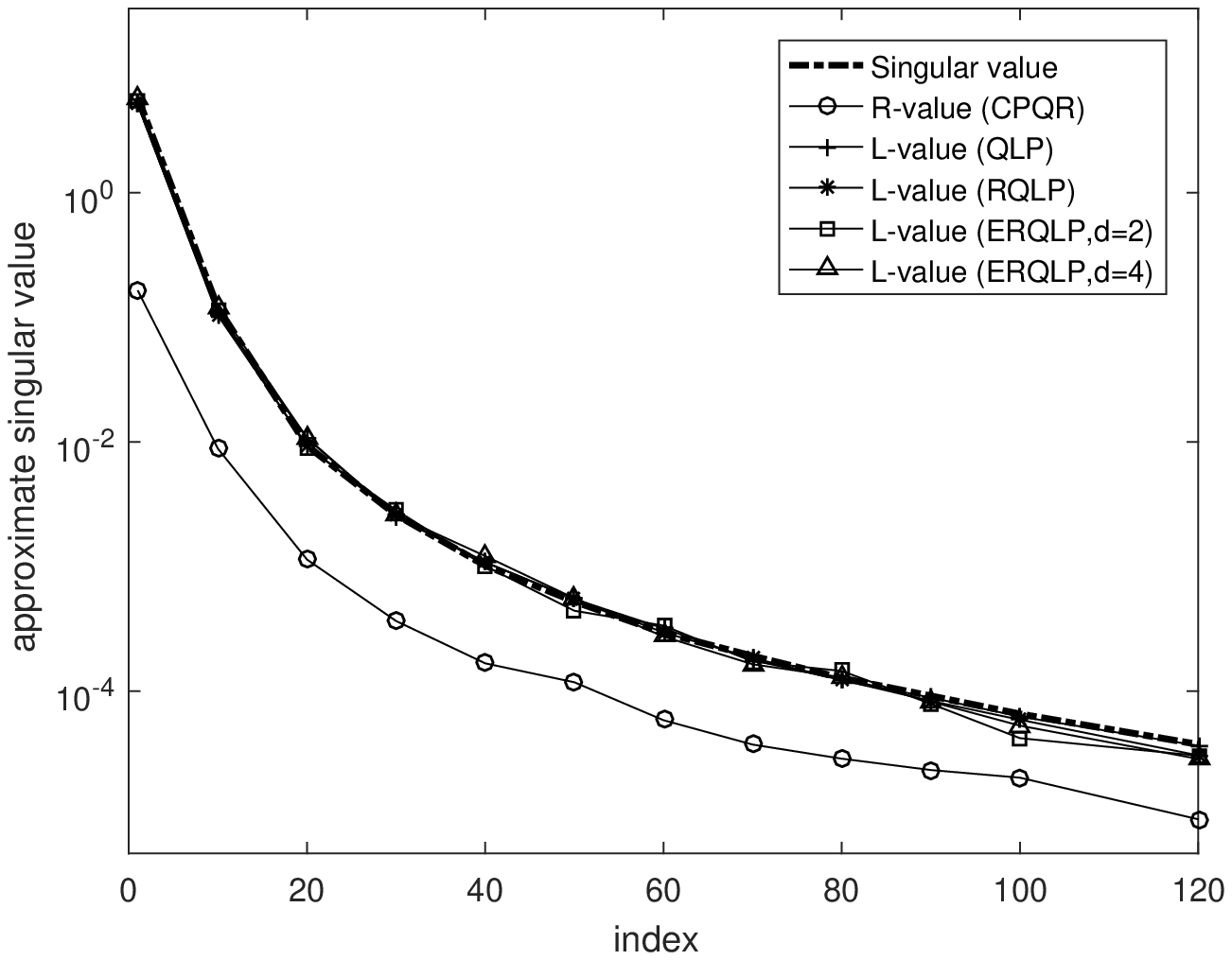}}
	\caption{The approximations to the singualr values of $A$ for the example {\sf heat} (left) and {\sf  phillips} (right). The matrix size is $4000 \times 4000$. }
	\label{fig2} 
\end{figure}

\section{Concluding remarks}\label{sec6}\noindent
Based on the randomized range finder algorithm, two randomized QLP decomposition are proposed: RQLP and ERQLP, which can be used for producing standard low rank factorization, like a partial QR decomposition or a partial singular value decomposition and have several attractive properties.
\begin{enumerate}[\sffamily 1.]
\addtolength{\itemsep}{ -0.5 em} 
     \item The theoretical cost of the implementation of  RQLP and ERQLP only need $\mathcal{O}(mnk)$.
     \item  It is easy to convert the RQLP algorithm to a block scheme.
     \item  The $L$-values from ERQLP can track the singular values of $A$  better than QLP, in particular, much better than the $R$-values given by CPQR. The computational time of randomized QLP variants  is much less than the original deterministic one.  And we provide mathematical justification and numerical experiments for these phenomena.
   \end{enumerate}
Moreover, the randomized QLP decomposition are suitable for the {\it fixed-rank} determination problems in practice. Different from \eqref{fixedrank}, we can seek $V$ and $B$ with a suitable rank such that
\begin{equation*}
  \| E \|_F^2 =\| A\|_F^2  - \|B \|_F^2  \leq  \varepsilon,
\end{equation*}
where $\varepsilon$ is the accuracy tolerance, and set $\| E \|_F$  as the error indicator which leads to a {\it auto-rank} randomized QLP  algorithm. That is  to say,  the rank of factor matrices needs to be automatically determined by some accuracy conditions.  We will continue to further in-depth study from the viewpoint of both theory and computations for this problem in future.

\end{document}